\documentclass[11pt]{amsart}

\usepackage{amsfonts,amssymb,stmaryrd,amscd,amsmath,latexsym,amsbsy,color, hyperref}

\usepackage{amssymb}
\usepackage{amsfonts}
\usepackage{latexsym}
\usepackage{bbold}
\usepackage{tabu}
\usepackage{ulem}
\usepackage[footskip=0.4in]{geometry}

\usepackage[notcite,notref,final]{showkeys}

\newtheorem{theorem}{Theorem}[section]
\newtheorem{lemma}[theorem]{Lemma}
\newtheorem{proposition}[theorem]{Proposition}
\newtheorem{corollary}[theorem]{Corollary}
\theoremstyle{definition}
\newtheorem{definition}[theorem]{Definition}

\newtheorem{example}[theorem]{Example}

\newtheorem{remark}[theorem]{Remark}

\newcommand{\Tr}{\text{Tr}}
\newcommand{\id}{\text{id}}

\newcommand{\ben}{\begin{enumerate}}
\newcommand{\een}{\end{enumerate}}

\theoremstyle{plain}

\newtheorem*{sol}{Solution}

\theoremstyle{definition}

\theoremstyle{remark}

\newcommand{\solu}[1]{\begin{sol}{\bf (\ref{#1})}}

\begin{document}

\title[Finite dimensional Hopf actions on quantizations]{Finite dimensional Hopf actions on deformation quantizations}

\author{Pavel Etingof}
\address{Department of Mathematics, Massachusetts Institute of Technology,
Cambridge, MA 02139, USA}
\email{etingof@math.mit.edu}

\author{Chelsea Walton}
\address{Department of Mathematics, Temple University, Philadelphia, Pennsylvania 19122
USA}\email{notlaw@temple.edu}


\maketitle

\begin{abstract}
We study when a finite dimensional Hopf action on a quantum formal deformation $A$ of a commutative domain $A_0$ (i.e., a deformation quantization) must factor through a group algebra. In particular, we show that this occurs when the Poisson center of the fraction field of $A_0$ is trivial.
\end{abstract}

\section{Introduction} 

Throughout the paper, we will work over an algebraically closed field $k$ of characteristic zero. 
Let us say that an associative algebra $B$ has {\it No Finite Quantum Symmetry} ({\sf NFQS}) 
if any action of a finite dimensional Hopf algebra $H$ on $B$ factors through a group algebra, and 
has {\it No Semisimple Finite Quantum Symmetry} ({\sf NSFQS}) if this holds for semisimple Hopf actions. 
In previous papers (\cite{CEW1, CEW2, EGMW, EW1}), we and coauthors established 
these properties for various classes of algebras. In particular, in \cite{EW1} we 
proved the {\sf NSFQS} property when $B=:A_0$ is a commutative domain. 

The aim of this work is to investigate when these properties hold for Hopf actions on {\it quantum formal deformations} $A$ of a commutative domain $A_0$. To do so, we use the Poisson structure on $A_0$ and on its fraction field $Q(A_0)$, which are induced by the multiplication of $A$. 
Namely, we show that if the Poisson center of $Q(A_0)$ is trivial, then the {\sf NFQS} property holds. We summarize our main results in the table below, along with recalling related results in the literature.

\medskip

\begin{center}
\footnotesize{
\tabulinesep=1mm
\begin{tabu}{|c|c|c|c|c|}
\hline
Property & module algebra $B$ &
 \begin{tabular}{c}$H \curvearrowright B$ preserves\\ filtration of $B$? \end{tabular}& 
 \begin{tabular}{c}Poisson center \\of $Q(A_0)$ triv.? \end{tabular}
 & Reference\\
 \hline
 \hline
 {\sf NSFQS} & $A_0$ (commutative domain)  & not required & not required & \cite[Thm~1.3]{EW1}\\
{\sf NSFQS} & filtered deformation $\widetilde{A}$ of $A_0$  & sufficient & not required& \cite[Prop~5.4]{EW1}\\
 {\sf NFQS} & ${\bf A}_n(k)$ (Weyl algebra)  & not required & not required & \cite[Thm~1.1]{CEW2}\\
 {\sf NSFQS} & ${\bf A}_n(k[z_1,\dots,z_s])$  & not required & not required & \cite[Prop~4.3]{CEW1}\\
 {\sf NFQS} & $D(X)$ (algebra of diff'l ops)  & not required & not required & \cite[Thm~1.2]{CEW2}\\
 \hline
 \hline 
 {\sf NSFQS} & quantum deform'n $A$ of $A_0/_{k[[\hbar]]}$  & not required & not required & {\bf Proposition~\ref{ssresult}}\\
{\sf NFQS} & quantum deform'n $A$ of $A_0/_{k[[\hbar]]}$  & not required & sufficient & {\bf Theorem~\ref{maint}}\\
{\sf NFQS} & filtered deformation  $\widetilde{A}$ of $A_0$  & sufficient & sufficient & {\bf Corollary~\ref{corol}}\\
 \hline
\end{tabu}
}
\end{center}
\medskip

\begin{center} 
\footnotesize{\textsc{Table 1}. Various  settings for  No (Semisimple) Finite Quantum Symmetry, including our {\bf main results} here}
\end{center}


\section{Preliminaries} \label{prelim}
In this section, we recall the basic terminology pertaining to deformations of $k$-algebras, including {\it quantum deformations} of commutative algebras. We also discuss localizations of such quantum deformations. The section ends with material on {\it inner-faithful} Hopf actions.

\subsection{Deformations} Let us introduce the following definitions.

\begin{definition}[$A$, $A_N$] \label{def1} Let $A_0$ be an arbitrary $k$-algebra.
\begin{enumerate} 
\item[(a)]  A {\it (flat) formal deformation} of $A_0$ is a $k[[\hbar]]$-algebra $A$ which is topologically free over $k[[\hbar]]$ (i.e., $A \cong A_0[[\hbar]]$ as $k[[\hbar]]$-modules) and equipped with an algebra isomorphism $A/\hbar A\cong A_0$.  

\item[(b)] Given an non-negative integer $N$,  we say that a {\it (flat) $N$-th order deformation} of 
$A_0$ is a $k[\hbar]/(\hbar^{N+1})$-algebra $A_N$ which is free 
as a $k[\hbar]/(\hbar^{N+1})$-module and 
equipped with an algebra isomorphism $A_N/\hbar A_N \cong A_0$.  

\item[(c)] If, further, $A_0$ is a commutative $k$-algebra, then the not necessarily commutative algebras $A$ and $A_N$ above are referred to as {\it quantum deformations} of $A_0$.
\end{enumerate}
\end{definition}

Clearly, if $A$ is a formal deformation of $A_0$, then $A/\hbar^{N+1}A$ is an $N$-th order deformation of $A_0$ for any $N\ge 0$, and $A=\underleftarrow{\rm lim} (A/\hbar^{N+1}A)$. Thus, formal deformations may be viewed as {\it deformations of infinite order}. 

Given  a  Hopf algebra $H_0$,  a {\it formal deformation} $H$ and an {\it $N$-th order deformation} $H_N$ of $H_0$ are defined similarly to Definition~\ref{def1}. 
\smallskip

\begin{definition}[$\widetilde{A}$]  \label{def2} Let $A_0$ be a graded $k$-algebra.
A $\Bbb Z_+$-filtered algebra $\widetilde{A} = \bigcup_{n \geq 0} F^n \widetilde{A}$ is a {\it $\Bbb Z_+$-filtered deformation} of $A_0$ if we are given an isomorphism gr$_F\widetilde{A} \cong A_0$ as graded $k$-algebras. (The algebra $\widetilde{A}$ is also called a {\it PBW deformation} of $A_0$.)
\end{definition}

Any $\Bbb Z_+$-filtered deformation $\widetilde{A}=\bigcup_{n \geq 0} F^n \widetilde{A}$ 
of a graded algebra $A_0$ gives rise to its formal deformation via the Rees algebra construction.
 
\begin{definition}[$R(\widetilde{A})$, $\widehat{R}(\widetilde{A})$] With the notation above, the {\it Rees algebra} $R(\widetilde{A})$ is $\bigoplus_{n\ge 0}\hbar^n F^n\widetilde{A}$ and the {\it completed Rees algebra} 
$\widehat{R}(\widetilde{A})$ is $\prod_{n\ge 0}\hbar^n F^n\widetilde{A}$. 
\end{definition}

Clearly, $R(\widetilde{A})$ carries a grading, and is the span of the homogeneous elements of $\widehat{R}(\widetilde{A})$.  
Thus, $A:=\widehat{R}(\widetilde{A})$ is a homogeneous formal deformation of $A_0$ with $\deg(\hbar)=1$. 
Note also that $\widetilde{A}$ with its filtration can be recovered from $R(\widetilde{A})$ by the formula $\widetilde{A}=R(\widetilde{A})/(\hbar-1)$. 
In fact, any homogeneous formal deformation $A$ of $A_0$ gives rise to a $\Bbb Z_+$-filtered deformation via 
$\widetilde{A}=A_{\rm hom}/(\hbar - 1)$, where $A_{\rm hom}$ is the span of the homogeneous elements of $A$.  

\smallskip

Now take $A_0$ to be a commutative $k$-algebra. Suppose $A$ is a quantum $N$-th order deformation of $A_0$ for $1\le N\le \infty$. 
Define the bilinear map $\lbrace{\hspace{.035in},\rbrace}: A_0\times A_0\to A_0$
 as follows: for any $a_0,b_0\in A_0$, let $\lbrace{a_0,b_0\rbrace}$ be the image 
of $[a,b]$ in $\hbar A/\hbar^2A\cong A_0$, where $a, b$ are any lifts of $a_0,b_0$ to $A$. (This map is well defined since $A_0$ is commutative.) 
It is well known that $\lbrace{\hspace{.035in},\rbrace}$ is a derivation in each argument, 
which is a Lie bracket (i.e., a Poisson bracket) if $N\ge 2$.

\begin{definition} Given $A_0$, a commutative $k$-algebra with Poisson structure as above, we say that the $N$-th order quantum deformation
$A$ of $A_0$ is an $N$-th order {\it deformation quantization} of the Poisson algebra $(A_0,\lbrace{\hspace{.035in},\rbrace})$.
(If we do not specify the order, then we mean that $N=\infty$.)
\end{definition}

\begin{example} \label{ex1} (1) Take $A_0=k[x,y]$ with Poisson bracket $\{y,x\} = 1$. Then, the Weyl algebra $\bold A_1(k) = k\langle x,y \rangle/(yx-xy-1)$ is a filtered deformation of $A_0$ (with $\deg(x)=0$, $\deg(y)=1$), and gives rise to the quantum formal deformation $A=k[x,y][[\hbar]]$ of $A_0$ with multiplication defined by 
the Moyal formula 
$$
f*g= \sum_{i\ge 0}\frac{\hbar^i}{i!}\partial_y^if \cdot \partial_x^ig.
$$

\noindent (2)  Take $A_0=k[x_1,\dots,x_n]$ with $\{x_i,x_j\} = \lambda_{ij}x_i x_j$, $\lambda_{ij}\in k$. Let $q_{ij}\in 1+\hbar\lambda_{ij}+O(\hbar^2)\in k[[\hbar]]$,
with $q_{ij}q_{ji}=1$. Then, the $\hbar$-adically completed quantum polynomial algebra $A$ generated by $x_1,\dots,x_n$ with relations $x_i x_j = q_{ij} x_j x_i$
 is a quantum formal deformation of $A_0$. 
\medskip

\noindent (3) Take a Lie algebra $\mathfrak{g}$ and let $A_0$ be the symmetric algebra $S(\mathfrak{g})$, with $\{x,y\} = [x, y]_{\mathfrak g}$ for $x,y\in {\mathfrak g}$. 
Then, the enveloping algebra $U(\mathfrak{g})$ is a $\mathbb{Z}_+$-filtered deformation  of $A_0$.
\medskip

\noindent (4) Let $X$ be an abelian variety over $k$, ${\mathcal L}$ be an ample line bundle on $X$, and $\sigma\in \text{Aut}(X(k[[\hbar]]))$ be such that $\sigma=\id\text{ mod }\hbar$. Define the line bundles $\mathcal{L}_n:=\mathcal{L} \otimes \mathcal{L}^\sigma \otimes \cdots \otimes \mathcal{L}^{\sigma^{n-1}}$ on $X$ (with $\mathcal{L}_0~:=~\mathcal{O}_X$). Take $A:=B(X,{\mathcal L},\sigma) =\widehat {\bigoplus}_{n \geq 0} H^0(X, \mathcal{L}_n)$, the $\hbar$-adically completed twisted homogeneous coordinate ring of $X$ (\cite{ATV}). Given an ample line bundle $\mathcal{E}$ on $X$, we have that $\dim H^0(X, \mathcal{E})$ equals the Euler characteristic of $\mathcal{E}$, and hence is deformation-invariant. Therefore, $A$ is a torsion-free, separated, and $\hbar$-adically complete $k[[\hbar]]$-module such that $A/\hbar A=A_0$, i.e., $A\cong A_0[[\hbar]]$ as a $k[[\hbar]]$-module (since a similar statement holds for every homogeneous component of $A$).
Therefore,  $A$ is a quantum formal deformation of a homogeneous coordinate ring $A_0 := \bigoplus_{n \geq 0} H^0(X, \mathcal{L}^{\otimes n})$.
\end{example}


\subsection{Localization of quantum deformations}

\begin{lemma}\label{ore} 
Let $A_0$ be a commutative domain, and let $A_N$ be an $N$-th order quantum  deformation of $A_0$, for $N<\infty$.  Take $S$ to be the set of all regular elements of $A_N$ \textnormal{(}i.e., $S=A_N \setminus \hbar A_N$\textnormal{)}.  Then,
\begin{enumerate}
\item there exists the classical quotient ring $Q(A_N)=S^{-1}A_N$, 
\item $Q(A_N)$ is an $N$-th order deformation of the quotient field $Q(A_0)$, and 
\item $Q(A_N)$ is both left and right Artinian.
\end{enumerate}
\end{lemma}

\begin{proof} 
To prove (1), we show that $S$ satisfies both the right and left Ore conditions. Let $a\in A_N$ and $s\in S$. 
Note that ${\rm ad}(s)(a) \in \hbar A$, and so ${\rm ad}(s)^{N+1}a=0$. 
Hence,
$$
s^{N+1}a=\left(\textstyle \sum_{j=0}^N s^{N-j}{\rm ad}(s)^j(a)\right)s,
$$
and $S$ satisfies the left Ore condition. 
The right Ore condition is proved similarly. 
Now (1) follows from Ore's theorem. 

Part (2) follows easily from (1), and (3) follows immediately from (2).
\end{proof} 

Now let $A$ be a quantum formal deformation of $A_0$ (i.e., a deformation of infinite order). Define 
$$
Q(A):=\underleftarrow{\rm lim}~Q(A/\hbar^{N+1}A).
$$

\begin{example} \label{ex:field} If  $A_0$ is a field,    
then $Q(A)=A$ since all elements not in $\hbar A$ are already invertible. 
Therefore, $A[\hbar^{-1}]$ is a division algebra. 
\end{example}

\subsection{Inner-faithful Hopf actions} Recall that a Hopf algebra $H$ {\it acts on an algebra} $B$ ({\it from the left}) if $B$ is a (left) $H$-module algebra, or equivalently, if $B$ is an algebra object in the category of (left) $H$-modules.

\begin{definition}
We say that an action of a Hopf algebra $H$ on an algebra $B$ is {\it inner-faithful} if there does not exist a nonzero Hopf ideal of $H$ that annihilates the $H$-module $B$.
\end{definition}

One can always pass to an inner-faithful Hopf action by considering an action of a quotient Hopf algebra. 

We will need the following  auxiliary result; the standard proofs are omitted.

\begin{lemma} \label{lem:infaith}
Let $H$ be a finite dimensional Hopf algebra.
\begin{enumerate}
\item Suppose that $H$ acts on a $\mathbb{Z}_+$-filtered algebra $\widetilde{A} = \bigcup_{n\geq 0} F^n \widetilde{A}$ so that $F_n \widetilde{A}$ is $H$-stable for all $n \geq 0$. Then, there is an induced $H$-module algebra structure on gr$_F \widetilde{A}$ given by $h \cdot \overline{a} = \overline{(h \cdot a)}_n$ where $a \in F^n \widetilde{A}$ is any lift of $\overline{a} \in F^n \widetilde{A}/F^{n-1}\widetilde{A}$. Also, there is an induced $H$-action on the Rees algebra ${R}(\widetilde{A})$ and  the completed Rees algebra $\widehat{R}(\widetilde{A})$ so that $\hbar^nF^n \widetilde{A}$ is $H$-stable for all $n \geq 0$; this action is inner-faithful if and only if the given $H$-action on $\widetilde{A}$ is inner-faithful.

\smallskip

\item Suppose $H$ acts on a formal deformation $A$ of an algebra $A_0$. If the action of $H$ on $A_0$ 
is inner-faithful, then so is the $H$-action on $A$. The converse holds if $H$ is semisimple.

\end{enumerate}
\end{lemma}

\begin{proof}
We will only prove (2). If $I\subset H$ is a Hopf ideal annihilating $A$, then it clearly annihilates
$A_0$, implying the forward direction. 
The converse follows from the following standard fact: if $H$ is 
a semisimple algebra and $V$ a formal deformation of an $H$-module $V_0$ then 
$V$ is isomorphic to $V_0[[\hbar]]$ as an $H$-module.    
\end{proof}

\begin{remark} The converse in Lemma \ref{lem:infaith}(2) may fail if $H$ is not semisimple, as shown by  
 \cite[Example~3.2(d)]{CWWZ}.
 \end{remark} 


\section{The main results} \label{mainresults}

In this section we present the main results, including the results highlighted in Table~1, along with Theorem~\ref{fieldthm} which is needed for the proof of Theorem~\ref{maint}. The proof of Theorem~\ref{fieldthm} is postponed to the next section.

First, we obtain the following generalization of  \cite[Proposition 5.4]{EW1}.

\begin{proposition} \label{ssresult}
If $H_0$ is a semisimple Hopf algebra and $A_0$ is a commutative domain, then the action of $H_0$
 on a quantum formal deformation $A$ of $A_0$ factors through a group action.  
 \end{proposition}  

\begin{proof}
Without loss of generality, we may assume that the $H_0$-action on $A$ is inner-faithful. 
Since $H_0$ is semisimple, by Lemma \ref{lem:infaith}(2) the induced action of $H_0$ on $A_0$ is inner-faithful. Hence, $H_0$ is a finite group algebra  by \cite[Theorem~1.3]{EW1}. \end{proof} 

We would like to generalize this result to the case when $H_0$ is not necessarily semisimple and, still more generally, to the case when we have an action of a formal deformation $H$ of a finite dimensional Hopf algebra $H_0$. In this case, {\it nontrivial} actions of $H_0$ on a commutative domain $A_0$ (that is, ones not factoring through a group action) are possible; see e.g., \cite{EW2}. We want to see when these actions can lift to actions of $H$ on $A$. 

Recall that $A_0$ carries a Poisson bracket induced by the deformation $A$, and by virtue of being a biderivation, this bracket extends uniquely to the quotient field $Q(A_0)$. The following theorem shows that a nontrivial action of $H_0$ on $A_0$ cannot lift if the induced Poisson bracket 
on the fraction field $Q(A_0)$ has trivial center; the proof is presented in Section~\ref{pf3.2}.

\begin{theorem}\label{fieldthm} Let $H$ be a formal deformation 
of a finite dimensional Hopf algebra $H_0$ which acts on a quantum formal deformation 
$A$ of a commutative domain $A_0$. If the Poisson center of $Q(A_0)$ 
is trivial \textnormal{(}i.e., $\lbrace f,g\rbrace=0$ for all $g\in Q(A_0)$ implies $f\in k$\textnormal{)}, 
then the induced action of $H_0$ on $A_0$ 
factors through a group action. 
\end{theorem} 

Using Theorem \ref{fieldthm}, we prove our main result, which is the following theorem.  

\begin{theorem}\label{maint} Let $H_0$ be a finite dimensional Hopf algebra 
which acts on a quantum formal deformation 
$A$ of a commutative domain $A_0$. If the Poisson center of $Q(A_0)$ 
is trivial, then the action of $H_0$ on $A$ 
factors through a group action. 
\end{theorem} 

\begin{proof} Without loss of generality, we may assume that the action of $H_0$ on $A$ is inner-faithful. 

Let $I$ be the annihilator of $A_0$ as an $H_0$-module, i.e., the 
set of $x\in H_0$ such that $xA\subset \hbar A$. The action of  $H:=H_0[[\hbar]]$ (the trivial deformation) 
on $A$ satisfies the assumptions of Theorem \ref{fieldthm}. 
Thus, by Theorem \ref{fieldthm}, the action of $H_0$ on $A_0$ factors through a group algebra; in other words, $H_0/I=kG$ for some finite group $G$. 
In particular, $I$ is a Hopf ideal. Then, $I^\infty:=\bigcap_{m\ge 0}I^m$ is a Hopf ideal in $H_0$ acting trivially on $A$. So $I^\infty=0$ by inner-faithfulness. Hence, there is $r>0$ 
such that $I^r=0$; let us take the smallest such $r$. 
Since $I$ is a nilpotent ideal and $H_0/I$ is semisimple, we get that $I={\rm Rad}(H_0)$. So the radical of $H_0$ is a Hopf ideal. 

Our job is to show that $I$ acts by zero on $A$ (then it would follow that $H_0=kG$). Assume the contrary. Let $s$ be the largest integer such that $IA\subset \hbar^s A$
(it exists since we have assumed that $IA\ne 0$). Consider $H':=\sum_{m=0}^{r-1} \hbar^{-ms}I^m[[\hbar]]\subset H[\hbar^{-1}]$ (where $I^0=H_0$); it is the Rees algebra of $H_0$ with respect to the decreasing filtration by powers of $I$, with $\deg(I)=s$. Since $I$ is a Hopf ideal, we have $\Delta(I)\subset H\otimes I+I\otimes H$. Hence 
$$\Delta(\hbar^{-ms}I^m)\subset \sum_{p+q=m}(\hbar^{-mp}I^p)\otimes (\hbar^{-mq}I^q),$$ so $\Delta(H')\subset H'\otimes H'$, and we obtain that $H'$ is a Hopf algebra. Furthermore, $H'$ is a formal deformation of the Hopf algebra 
${\rm gr} H_0:=\bigoplus_{m=0}^{r-1}I^m/I^{m+1}$, 
the associated graded algebra of $H_0$ under the radical filtration
(which, in this case, is a Hopf algebra filtration, as Rad($H_0$) is a Hopf ideal of $H_0$). 
Moreover, by definition $H'$ acts on $A$. Hence ${\rm gr} H_0$ acts 
on $A_0$ by reducing modulo $\hbar$. 

By Theorem \ref{fieldthm}, the action of ${\rm gr} H_0$ on $A_0$ must factor through a group algebra. In particular, the radical ${\rm gr} I$ (which is a Hopf ideal of ${\rm gr} H_0$) acts by zero on $A_0$. 

On the other hand, by our assumption, there exists $x\in I$ and $a\in A$
such that $xa=\hbar^s b$, where $b$ has a nonzero image $b_0$ in $A_0$. 
Then, $(\hbar^{-s}x)a=b$. So, denoting by $x_0$ the image of $\hbar^{-s}x\in \hbar^{-s}I\subset H'$ in ${\rm gr} I\subset {\rm gr} H_0$, and denoting by $a_0$ the image of $a$ in $A_0$, we obtain $x_0a_0=b_0\ne 0$. 
This means that ${\rm gr} I$ acts by nonzero on $A_0$, a contradiction. 
The theorem is proved. 
\end{proof} 

\begin{corollary}\label{corol}  Let $\widetilde{A}$ be a $\Bbb Z_+$-filtered algebra such that $A_0={\rm gr} \widetilde{A}$
is a commutative domain. Suppose that a finite dimensional 
Hopf algebra $H$ acts on $\widetilde{A}$ preserving the filtration of $\widetilde{A}$. If the Poisson center of $Q(A_0)$ is trivial, then the action of $H$ factors through a group action. 
\end{corollary} 

\begin{proof} Without loss of generality, we assume that  $H$ acts on $\widetilde{A}$ inner-faithfully. Since the $H$-action on $\widetilde{A}$ preserves the filtration of $\widetilde{A}$, it extends to an inner-faithful $H$-action on the completed Rees algebra $\widehat{R}(\widetilde{A})$ by Lemma~\ref{lem:infaith}(1). Now $H$ is a finite group algebra by Theorem~\ref{maint}.
\end{proof}

\begin{remark} Suppose that $A_0$ is a finitely generated commutative domain, that is, $A_0={\mathcal O}(X)$, the algebra of regular functions on some irreducible affine variety $X$ over $k$. Then, the condition that the Poisson center of $Q(A_0)=k(X)$ is trivial holds, in particular, when the induced Poisson bracket 
on $X$ is generically symplectic (i.e., there exists a dense smooth affine open set $U\subset X$ and a closed nondegenerate 2-form $\omega$ on $U$ such that $\lbrace{f,g\rbrace}=(df\otimes dg,\omega^{-1})$ for any $f,g\in {\mathcal O}(X)$). For example, one may take $X$ to be any affine symplectic variety, 
and $A$ a deformation quantization of $\mathcal{O}(X)$ (e.g., Fedosov's quantization); see \cite{BK}. 
\end{remark} 

\begin{example} The condition  in Theorem \ref{fieldthm} and Theorem \ref{maint} that the Poisson center of $Q(A_0)$ is trivial cannot be replaced by a weaker condition 
that the Poisson center of $A_0$ is trivial. For example, 
consider the quantum polynomial algebra $A$ with generators 
$x,y,z$ and relations $xy=qyx$, $xz=qzx$, $zy=qyz$, 
where $q=\exp(\hbar)$. Then, the induced Poisson bracket on 
$A_0=k[x,y,z]$ is given by $\lbrace{x,y\rbrace}=xy$, 
$\lbrace{z,y\rbrace}=yz$, $\lbrace{x,z\rbrace}=xz$, 
and it is easy to see that the Poisson center 
of $A_0$ is trivial. On the other hand, the Poisson center of $Q(A_0)$ contains the element $xy/z$. 

Let $H_0$ be the Sweedler Hopf algebra 
with grouplike generator $g$ such that $g^2=1$ and $(1,g)$-skew-primitive generator $a$ such that $ga=-ag$ and $a^2=0$. Define an action of $H_0$ on $A$ by 
$$
g\cdot x=x,\quad g\cdot y=y,\quad g\cdot z=-z,\quad  a\cdot x=0,\quad a\cdot y=0, \quad a\cdot z=xy.
$$ 
It is easy to check that this action is well defined, and does not factor through a group algebra, even 
after reducing modulo $\hbar$. 
\end{example}


\section{Proof of Theorem \ref{fieldthm}} \label{pf3.2}

 Since $H$ acts on $A$,  it acts on $A/\hbar^{N+1}A$ for any $N$. Hence, $H$ acts on the classical quotient ring $Q(A/\hbar^{N+1}A)$ by \cite[Theorem~2.2]{SV},  and
by taking the inverse limit in $N$, we get an action of $H$ on 
$Q(A)$. Thus, without loss of generality we may assume that $A_0$ is a field.

One of the main steps of the proof is to show that many invariants 
in $A_0^{H_0}$ lift to invariants in $A^H$. Namely, 
let us say that an element $a_0\in A_0^{H_0}$ is a {\it liftable invariant} 
if there exists $a\in A^H$ equal to $a_0$ modulo $\hbar$. 
\medskip

\noindent {\it Notation} ($K$). Let $K \subset A_0$ be 
the subset (in fact, subfield) of liftable invariants under the action of $H_0$.

\begin{lemma}\label{fini} The field $A_0$ is an algebraic extension of $K$.
\end{lemma}  

\begin{proof}
Let $d :=\dim H_0 = \dim_{k((\hbar))} H[\hbar^{-1}]$. Let $D :=A[\hbar^{-1}]$, which is a division algebra over $k((\hbar))$ 
by Example~\ref{ex:field}. Further, $H[\hbar^{-1}]$ acts $k((\hbar))$-linearly 
on $D$. Thus, by \cite[Corollary~2.3]{BCF}, 
$D$ has dimension $\le d$ over $D^{H[\hbar^{-1}]}$ as a left vector space. 
Now let $x_0\in A_0$ and $x\in A$ be its lift to $A$. 
As $[D:D^{H[\hbar^{-1}]}]\le d$, we have that $x$ satisfies an equation 
\begin{equation}\label{poleq}
b_0x^n+b_1x^{n-1}+ \cdots+b_n=0,
\end{equation}
where $b_0=1$, $b_i\in D^{H[\hbar^{-1}]}$ and $n\le d$. Let $m$ be the smallest value of the $\hbar$-adic valuation of $b_i$ in $D$ (over all $i$); clearly, $m\le 0$. Projecting \eqref{poleq}
to $\hbar^mA/\hbar^{m+1}A$, we get a nontrivial equation 
\begin{equation} \label{algeb}
c_0x_0^s+c_1x_0^{s-1}+\cdots+c_s=0
\end{equation}
of possibly lower degree $s\le n$. Note that $c_i\in K$ by definition, so $x_0$ 
is algebraic over $K$. 
\end{proof} 

Now we proceed with the proof of Theorem~\ref{fieldthm}. Consider the Galois map 
$$\beta: A_0\otimes A_0\to A_0\otimes H_0^*, \quad f\otimes g \mapsto (f\otimes 1)\rho(g),$$ 
where $\rho: A_0\to A_0\otimes H_0^*$ 
is the coaction map. Then, $$B:={\rm Im}\beta$$ is a commutative coideal subalgebra 
in the Hopf algebra $A_0\otimes H_0^*$ (regarded as a finite dimensional Hopf algebra over $A_0$); the commutativity is clear and the coideal subalgebra condition follows from an argument similar to \cite[Lemma~3.2]{EW1}. Moreover, by \cite[Lemma~3.3]{CEW2} it suffices to show that

\bigskip

\noindent ($\dag$) \hspace{.6in}
\begin{tabular}{c}  $B$ is defined over $k$, that is,
$B=A_0\otimes B_0$, where $B_0$ is a subalgebra of $H_0^*$.
\end{tabular}

\bigskip

Let $\lbrace{h_i\rbrace}$ be a basis of $H_0$, and let $\lbrace{h_i^*\rbrace}$ be the dual basis 
of $H_0^*$. Then for $f \in A_0$
$$
\rho(f)= \textstyle \sum_{i=1}^d \rho_i(f)\otimes h_i^*,
$$ 
where 
$\rho_i: A_0\to A_0$. 

\begin{lemma}\label{poiscom} Suppose $a_0\in K$ is a liftable invariant. Then for any $f_0\in A_0$ and all $i$, one has 
$$
\rho_i(\lbrace{a_0,f_0\rbrace})=\lbrace{a_0,\rho_i(f_0)\rbrace}.
$$
\end{lemma} 
 
\begin{proof}
Let us fix an isomorphism $H\cong H_0[[\hbar]]$ as $k[[\hbar]]$-modules, 
and by abusing notation, denote the coaction of $H^*$ on $A$ also by $\rho$ and its components by $\rho_i$. Let $a$ be a lift of $a_0$ to $A^H$, 
and let $f$ be a lift of $f_0$ to $A$. We have 
$$
\rho_i([a,f])=[a,\rho_i(f)].
$$
Projecting this equation to $\hbar A/\hbar^2A \cong A_0$, we obtain the desired statement. 
\end{proof} 

Introduce the following notation. Let $r:=\dim B$, and $v_1,\dots,v_r$ be elements of $A_0$ such that $\rho(v_1),\dots,\rho(v_r)$ are linearly independent, and hence form a basis of $B$ over $A_0$. 
Let $h_1, \dots, h_d$ be a basis of $H_0$, and let $\bold B:=(b_{ij})$ be the matrix representing $B$ in the Grassmannian  \linebreak ${\rm Gr}_r(A_0 \otimes H_0^*)=:{\rm Gr}_r(d)$ of $r$-dimensional subspaces in a $d$-dimensional space with respect to these bases. Namely,  $\rho(v_i) = \sum_j b_{ij} \otimes h_j^*$ where $b_{ij}=\rho_j(v_i) \in A_0$.

Recall that the homogeneous coordinate ring of ${\rm Gr}_r(d)$ under the Pl\"ucker embedding is generated by the minors $\Delta_I$ of an $r$-by-$d$ matrix
attached to subsets $I\subset \lbrace{1,\dots,d\rbrace}$ with $|I|=r$. Pick $I$ so that $\Delta_I(\bold B)\ne 0$. 
Let $J\subset \lbrace{1,\dots,d\rbrace}$ with $|J|=r$ be such that $|J\cap I|=r-1$. Then, the Pl\"ucker coordinates 
$p_{IJ}:=\Delta_J/\Delta_I$ are rational functions on ${\rm Gr}_r(d)$ which form a local coordinate system near $\bold B$. 

Note that $B$ is defined over $k$ precisely when $B \in {\rm Gr}_r(H_0^*) \subset {\rm Gr}_r(A_0 \otimes H_0^*)$. So property $(\dag)$ is equivalent to the property that for all $J$, the ratios $p_{IJ}(\bold B)$ lie in 
$k$, which is what remains to be shown.  

To this end, let $a_0\in K$ be a liftable invariant. Since the vectors $\rho(v_i)$ form a basis of $B$, there exists an $r$-by-$r$ matrix $\bold C=(c_{im})$ with $c_{im}\in A_0$, such that 
$$
\rho(\lbrace{a_0,v_i\rbrace})=\textstyle \sum_m c_{im}\rho(v_m).
$$
By Lemma \ref{poiscom}, 
$$
\sum_j \{a_0, \rho_j(v_i)\} \otimes h_j^* = \sum_{m,j} c_{im} \rho_j(v_m) \otimes h_j^*. 
$$
So,
$$ 
\lbrace{a_0,b_{ij}\rbrace}=\textstyle \sum_m c_{im}b_{mj}. 
$$
This implies that $\lbrace{a_0,\Delta_I(\bold B)\rbrace}=\Tr(\bold C)\Delta_I(\bold B)$, and thus 
\begin{equation} \label{eq3}
\lbrace{a_0,p_{IJ}(\bold B)\rbrace}= \frac{1}{\Delta_I(\bold B)^2} \biggl(\Delta_I(\bold B) \{a_0, \Delta_J(\bold B) \} - \Delta_J(\bold B) \{a_0, \Delta_I(\bold B)\}\biggr) = 0. 
\end{equation}

Now by Lemma~\ref{fini}, any $f \in A_0$ satisfies an equation $c_0f^s + c_1f^{s-1} + \cdots + c_0 =0$ for some $c_i \in K$, with $s$ minimal. Since the Poisson bracket is a biderivation, we have 
$$
0  = \lbrace{\textstyle \sum_{i=0}^s c_{s-i}f^i, p_{IJ}(\bold B)\rbrace}
 \overset{\eqref{eq3}}{=} \left(\textstyle \sum_{i=1}^{s} ic_{s-i}f^{i-1} \right)\lbrace{f, p_{IJ}(\bold B)\rbrace}.$$
Since $s$ is minimal, $\textstyle \sum_{i=1}^{s} ic_{s-i}f^{i-1} \neq 0$. This implies that $\lbrace{f,p_{IJ}(\bold B)\rbrace}=0$ for any $f\in A_0$. Finally, since the Poisson center of $A_0$ is trivial, we obtain that $p_{IJ}(\bold B)\in k$. Theorem~\ref{fieldthm} is proved.

\begin{remark} One can generalize the main results of this article by replacing the induced Poisson bracket on $A_0$ with the {\it induced Poisson bracket of depth $m$} as follows.

Let $A$ be a noncommutative formal deformation of $A_0$, and let $m$ be the largest integer such that $[a,b]\in \hbar^m A$ for all $a,b\in A$. Given $a_0,b_0\in A_0$, pick lifts $a,b$ of $a_0,b_0$ to $A$, and consider the projection $\lbrace{a_0,b_0\rbrace}$ of $[a,b]$ to $\hbar^m A/\hbar^{m+1}A$. Then, it is well known that $\lbrace{~ ,\rbrace}$ is a nonzero Poisson bracket for $A_0$; let us call it the {\it induced Poisson bracket of depth $m$}. The same construction applies to filtered deformations, by passing to the completed Rees algebra. 

This generalizes the above setting, in which  $m=1$. More precisely, the usual 
induced Poisson bracket is the bracket of depth $1$. If it turns out to be zero,
then we can define the Poisson bracket of depth 2. If it also turns out to be zero, 
then we can define a Poisson bracket of depth 3, and so on, until we reach
some depth $m$ where the bracket is nonzero (which will necessarily happen if 
$A$ is noncommutative). 
  
Now Theorem \ref{fieldthm}, Theorem \ref{maint}, and Corollary \ref{corol} generalize  to this setting in a straightforward fashion, with the same proofs. In other words, if the Poisson center of $Q(A_0)$ with respect to 
a Poisson bracket of any depth $m$ is trivial, then the appropriate Hopf action must factor
through a group action.   
\end{remark}

\section*{Acknowledgments}
The authors were supported by the National Science Foundation: NSF-grants DMS-1502244 and DMS-1550306.

\end{document}